\documentclass[12pt]{amsart}
\usepackage{eurosym,graphicx,amscd,color,amsmath,amsfonts,amssymb,amsthm,centernot,geometry,hyperref,amsmath,amsfonts,amssymb,graphicx,mathtools}
\usepackage[initials]{amsrefs}
\usepackage[all]{xy}

\newtheorem{theorem}{Theorem}

\newtheorem{definition}{Definition}

\newtheorem{proposition}{Proposition}
\newtheorem{remark}{Remark}
\newtheorem{conjecture}{Conjecture}

\begin{document}
\title[On the set of functions  that vanish at infinity and have a unique maximum]{On the set of functions that vanish at infinity and have a unique maximum}

\author[Ara\'{u}jo]{G. Ara\'{u}jo}
\address[G. Ara\'{u}jo]{\mbox{}\newline \indent Departamento de Matem\'{a}tica \newline\indent Universidade Estadual da Para\'{\i}ba \newline\indent 58.429-500 Campina Grande, Brazil.}
\email{gustavoaraujo@servidor.uepb.edu.br}

\author[Barbosa]{A. Barbosa}
\address[A. Barbosa]{\mbox{}\newline \indent Departamento de Matemática \newline \indent Universidade Federal da Paraíba \newline\indent 58.051-900 João Pessoa, Brazil.}
\email{afsb@academico.ufpb.br}

\begin{abstract}
In this paper, we show that the set of continuous functions defined on $\mathbb{R}^n$ that approach zero at infinity and attain their maximum at precisely one (and only one) point is $n$-lineable but not $(n+2)$-lineable. This result complements some recent published works on an open question originally posed by Vladimir I. Gurariy (1935--2005) in 2003.
\end{abstract}

\keywords{Lineability; Spaceability; Linear spaces of continuous functions; Maximum}

\subjclass{15A03, 28A20, 46B87, 46E99, 60B99}

\thanks{The first author was supported by Grant 3024/2021, Para\'iba State Research Foundation (FAPESQ).
This study was financed in part by the Coordenação de Aperfeiçoamento de Pessoal de Nível Superior - Brasil (CAPES) - Finance Code 001}

\maketitle


\section{Notations and preliminaries}

Since its first appearance in 2005 (see \cite{AGS}, or check \cite{Studia2017} for a more modern reference), the concepts of lineability and spaceability have sparked interest of a large number of researchers. Moreover, just recently the American Mathematical Society incorporated these terms into its 2020 Mathematical Subject Classification, refereeing to them under the classifications 15A03 and 46B87. In essence, this idea revolves around identifying significant algebraic frameworks within non-linear subsets of a topological vector space, whenever possible. For a comprehensive overview of lineability, we refer to the recent monograph \cite{aron} or the expository paper \cite{BAMS2014}.

Let us briefly recall, here below, some terminology and notions on lineability and spaceability that shall be useful for our purpose in this manuscript. For any set $X$ we shall denote by $\mathrm{card}(X)$ the cardinality of $X$; we also define $\mathfrak{c} = \mathrm{card}(\mathbb{R})$ and $\aleph_0 = \mathrm{card} (\mathbb{N})$. Assume that $E$ is a vector space and $\beta\leq\dim(E)$ is any (finite or infinite) cardinal number. A subset $A \subset E$ is said to be \textit{$\beta$-lineable} if there exists a vector subspace $F$ of $E$ with $\dim(F) = \beta$ and $F\smallsetminus\{0\}\subset A$.

These notions were coined by V. I. Gurariy (1935-2005), and some early instances of findings in this field are also due to him. For instance, he  established that the collection of continuous nowhere differentiable functions on $\mathbb{R}$ contains, except for $\{0\}$, infinite-dimensional linear spaces. Moreover, in \cite{GQ}, Gurariy and Quarta showed the existence of two-dimensional spaces in $\hat{\mathcal{C}}(\mathbb{R})\cup\{0\}$, where $\hat{\mathcal{C}}(D)$ is the subset of $\mathcal{C}(D)$ consisting of bounded functions that attain their maximum at a single point. Here, as usual, for a topological space $D$, $\mathcal{C}(D)$ denotes the space of all continuous functions $f:D\longrightarrow\mathbb{R}$.

In the same article, Gurariy and Quarta claim not to be aware of the existence in $\hat{\mathcal{C}}(\mathbb{R})\cup\{0\}$ of spaces with dimensions greater than two. The answer to this question came only in 2020 in the article \cite{BHGJ}, where the authors, surprisingly, using tools from general topology, geometry, and complex analysis, not commonly found in the results of this theory, proved the non-$3$-lineability of $\hat{\mathcal{C}}(\mathbb{R})$.

In summary, from \cite{BHGJ,GQ} we have the following result:

\begin{theorem} \cite{BHGJ,GQ}
Let $F$ stand for a subspace of $\mathcal{C}(\mathbb{R})$ such that every nonzero function in $F$ attains its maximum at one (and only one) point. Then $\mathrm{dim}(F)\leq 2$. More generally, if $m\in \mathbb{N}$ and $F_m$ stands for a subspace of $\mathcal{C}(\mathbb{R})$ such that every nonzero function in $F_m$ attains its maximum at $m$ (and only $m$) points, then $\mathrm{dim}(F_m)\leq 2$. In other words, the subset of $\mathcal{C}(\mathbb{R})$ of functions attaining their maximum at $m$ (and only $m$) points is $2$-lineable but not $3$-lineable for every $m\in\mathbb{N}$.
\end{theorem}

In an attempt to answer the above question, Gurariy and Quarta, in the same article \cite{GQ}, studied the lineability of $\hat{\mathcal{C}}(\mathbb{R})$ restricted to functions that vanish at infinity. This additional property led to a simpler answer than the one given in \cite{BHGJ}.

The purpose of this article is to generalize this latter mentioned study for Euclidean spaces of dimension $n$.

\section{Preliminary Tools}

We shall denote by $\mathcal{C}_0(\mathbb{R})$ the subspace of ${\mathcal{C}}(\mathbb{R})$ consisting of functions such that $$\displaystyle \lim_{|x|\rightarrow\infty} f(x)=0,$$ and we denote by $\hat{\mathcal{C}}_0(\mathbb{R})$ the subset of $\mathcal{C}_0(\mathbb{R})$ consisting of all functions that attain their maximum at a unique point.

\begin{theorem}\label{gqc0} \cite{GQ}
$\hat{\mathcal{C}}_0(\mathbb{R})$ is $2$-lineable but not $3$-lineable.
\end{theorem}

In the following, we will use in $\mathbb{R}^n$ the norm $\|\cdot\|$ induced from the usual inner product $\left<\cdot,\cdot\right>$. Given $n\in\mathbb{N}$, we denote by $\mathcal{C}_0(\mathbb{R}^n)$ the subspace of ${\mathcal{C}}(\mathbb{R}^n)$ consisting of functions such that $\lim_{\|x\|\rightarrow\infty} f(x)=0$, and we define $\hat{\mathcal{C}}_0(\mathbb{R}^n)$ as the subset of $\mathcal{C}_0(\mathbb{R}^n)$ consisting of all functions that attain the maximum at only one point. Recall that since the functions in $\mathcal{C}_0(\mathbb{R}^n)$ converge to zero at infinity, the norm $\|f\|_\infty=\sup_{x\in\mathbb{R}^n}|f(x)|$ is well-defined in $\mathcal{C}_0(\mathbb{R}^n)$ and makes it a Banach space.

A natural step in advancing this theory would be to study the lineability of $\hat{\mathcal{C}}(\mathbb{R}^n)$. However, a detailed analysis of the solution presented in \cite{BHGJ} shows that the method used in \cite{BHGJ} apparently cannot be extended to the $n$-dimensional case.

On the other hand, since the lineability of $\hat{\mathcal{C}}_0(\mathbb{R})$ was obtained more naturally, it seems natural to start this study by generalizing Theorem \ref{gqc0}.

In the quest for an answer to the problem raised by Gurariy and Quarta in \cite{GQ}, the authors of \cite{GVDJD} proved the following result, which generalizes one of the three main results of \cite{GQ}:

\begin{theorem} \cite{GVDJD} \label{t1}
If $K$ is a compact subset of $\mathbb{R}^n$, and $V$ is a subspace of $\mathcal{C}(K)$ contained in $\hat{\mathcal{C}}(K)\cup\{0\}$, then $\dim(V)\leq n$.
\end{theorem}

In this work, we generalize two other results from \cite{GQ}, and as a result, we demonstrate that the set $\hat{\mathcal{C}}_0(\mathbb{R}^n)$ is $n$-lineable but not $(n+2)$-lineable.

Beforehand, let us state and prove some auxiliary results. The following Theorems \ref{t2} and \ref{t3} are well-known, but we will present them with their respective proofs as they will be useful in establishing the main results of this work.

Let us denote by $S^{n-1}$ the unitary sphere of $\mathbb{R}^n$, that is, $$S^{n-1}=\{x\in\mathbb{R}^n \ : \ \|x\|=1\}.$$

\begin{theorem} \cite{GVDJD} \label{t2}
Let $n\geq2$, and let $D$ be a topological space such that there exists a continuous bijection from $D$ to $S^{n-1}$. Then $\hat{\mathcal{C}}(D)$ is $n$-lineable. In particular, $\hat{\mathcal{C}}(S^{n-1})$ is $n$-lineable for every $n\geq 2$.
\end{theorem}

\begin{proof}
Let $\pi_i:S^{n-1}\longrightarrow \mathbb{R}$ be the projection onto the $i$-th coordinate, $i=1,\ldots,n$, and take $G:D\longrightarrow S^{n-1}$ a continuous bijection. Initially note that the functions $\pi_i$, $i=1,\ldots,n$, are linearly independent.


Let us verify that every non-trivial linear combination of the functions $\pi_i$, $i=1,\ldots,n$, has just one point of maximum. Let $f=\sum_{i=1}^na_i \pi_i$, with $a_i\neq0$ for some $i$, and $x\in S^{n-1}$. Then, $$f(x)=\sum_{i=1}^na_i \pi_i(x)=\left<a,x\right>,$$ where $a=(a_1,\ldots,a_n)$. Note that $f(x)\leq\|a\|\|x\|=\|a\|$. Moreover, observe that $\frac{a}{\|a\|}\in S^{n-1}$ and $$f\left(\frac{a}{\|a\|}\right)=\left<a,\frac{a}{\|a\|}\right>=\frac{\left<a,a\right>}{\|a\|}=\|a\|,$$ which means that $\frac{a}{\|a\|}$ is a maximum point of $f$. If $y\in S^{n-1}$ is another point of maximum of $f$, then $\left<a,y\right>=\|a\|=\|a\|\|y\|$, that is, $a$ and $y$ are linearly dependent. So $y=\frac{a}{\|a\|}$ or $y=-\frac{a}{\|a\|}$. Since $$f\left(-\frac{a}{\|a\|}\right)=-\|a\|<0<\|a\|=f\left(\frac{a}{\|a\|}\right),$$ we conclude that $y=\frac{a}{\|a\|}$, and consequently we get the unicity of the maximum point.

Let us consider now the set $$\{\pi_i\circ G \ : \ i=1,\ldots,n\}.$$ Since $G$ and $\pi_i$ are continuous, we conclude that $\pi_i\circ G$ is continuous for every $i$. If $$g=\sum_{i=1}^na_i(\pi_i\circ G)=\left(\sum_{i=1}^na_i\pi_i\right)\circ G=f\circ G$$ is a non-trivial linear combination, we have $g\neq 0$ due to the linear independence of the functions $\pi_i$, $i=1,\ldots,n$, and the fact that $G$ is a bijection. Furthermore, since $f$ attains its maximum at only in one point of $S^{n-1}$, it also follows from the bijectiveness of $G$ that $g$ attains its maximum at only in one point. This is sufficient to demonstrate that $\mathrm{span}\{\pi_i\circ G \ : \ i=1,\ldots,n\}$ is a vector space contained within $\hat{\mathcal{C}}(D)\cup\{0\}$.
\end{proof}


\begin{theorem} \cite{GVDJD} \label{t3}
Let $n\geq2$, and let $D$ be a topological space containing a closed set $Y$ such that there exists a continuous bijection $F:Y\longrightarrow S^{n-1}$ and a continuous extension of $F$, $G:D\longrightarrow\mathbb{R}^n$, with $\|G(x)\|<1$ for every $x\notin Y$. Then $\hat{\mathcal{C}}(D)$ is $n$-lineable.
\end{theorem}

\begin{proof}
Let $\pi_i$ be the projection onto the $i$-th coordinate of $\mathbb{R}^n$. Note that $$S^{n-1}=F(Y)=G(Y)\subset G(D) \subset\{x\in \mathbb{R}^n \ : \ \|x\|\leq1\}.$$ Let us see that any non-trivial linear combination $$f=\sum_{i=1}^na_i\pi_i:G(D)\longrightarrow\mathbb{R}$$ attains its maximum at a unique point $x_0\in G(D)$ and that this point necessarily belongs to $S^{n-1}$. Restricting $f$ to $S^{n-1}$, the same argument used in the proof of the previous theorem (Theorem \ref{t2}) shows us that there is unique $x_0\in S^{n-1}$ such that $f|_{S^{n-1}}$ reaches its maximum, that is, $$\sum_{i=1}^na_i\pi_i(x)\leq\sum_{i=1}^na_i\pi_i(x_0)$$ for all $x\in S^{n-1}$. Since for each $x\in S^{n-1}$, $-x$ also belongs to $S^{n-1}$ and taking into account that each $\pi_i$ is linear, we have that $$-\sum_{i=1}^na_i\pi_i(x)=\sum_{i=1}^na_i\pi_i(-x)\leq\sum_{i=1}^na_i\pi_i(x_0).$$ Therefore, we can conclude that for all $x\in S^{n-1}$ it is worth
\begin{equation}\label{eq}
\left|\sum_{i=1}^na_i\pi_i(x)\right|\leq\sum_{i=1}^na_i\pi_i(x_0).
\end{equation}
We now have to show that if $y\in G(D)\smallsetminus S^{n-1}$, then
\begin{equation*}
\sum_{i=1}^na_i\pi_i(y)<\sum_{i=1}^na_i\pi_i(x_0).
\end{equation*}
If $y=0$, the inequality is trivial. Let us then consider $y\neq0$. As $\|y\|<1$ by hypothesis, using \eqref{eq} we have
\begin{align*}
\sum_{i=1}^na_i\pi_i(y)&\leq\left|\sum_{i=1}^na_i\pi_i(y)\right|\\
&=\|y\|\cdot\left|\sum_{i=1}^na_i\pi_i\left(\frac{y}{\|y\|}\right)\right|\\
&<\left|\sum_{i=1}^na_i\pi_i\left(\frac{y}{\|y\|}\right)\right|\\
&\leq\sum_{i=1}^na_i\pi_i(x_0).
\end{align*}
From the previous calculations, we have that the function $h=\sum_{i=1}^na_i(\pi_i\circ G)$ reaches its maximum at $z=F^{-1}(x_0)$, and this point is unique, as $z_1\neq z$ implies $G(z_1)\neq G(z)$ and consequently $h(z_1)<h(z)$. Now, as $S^{n-1}\subset G(D)$ and the functions $\pi_i$, $i=1,\ldots,n$, are linearly independent in $S^{n-1}$, it follows that $\pi_i\circ G$, $i=1,\ldots,n$, are linearly independent in $D$. Taking the vector space $\mathrm{span}\{\pi_i\circ G \ : \ i=1,\ldots,n\}$ we obtain the result.
\end{proof}

\section{Main result}

Below, we will present some results that will be useful in proving the main result of this work. For $r>0$, let $B_r=\{x\in\mathbb{R}^n \ : \ \|x\|\leq r\}$.

\begin{theorem}\label{c1}
Let $n\geq 1$. The set $\hat{\mathcal{C}}_0(\mathbb{R}^n)$ is $n$-lineable.
\end{theorem}

\begin{proof}
Let $Y = S^{n-1}$ and $F$ be the identity function defined on $Y$. Consider $G:\mathbb{R}^n\longrightarrow \mathbb{R}^n$ defined as
\[
G(x)=
\begin{cases}
x & \text{if}\ x\in B_1,\\	
\frac{x}{\|x\|^2} & \text{if}\ x\in\mathbb{R}^n\smallsetminus B_1.
\end{cases}
\]
Evidently $G$ is a continuous extension of $F$ with $\|G(x)\|<1$ for every $x\notin Y$ (moreover, observe that $\lim_{\|x\|\rightarrow\infty}\|G(x)\|=0$). Therefore, from Theorem \ref{t3}, we can conclude that $\hat{\mathcal{C}}(\mathbb{R}^n)$ is $n$-lineable. More precisely, analyzing the proof of Theorem \ref{t3}, we see that $\mathrm{span}\{\pi_i\circ G \ : \ i=1,\ldots,n\}$ has dimension $n$ and is contained in $\hat{\mathcal{C}}(\mathbb{R}^n)\cup\{0\}$, where $\pi_i$ is the projection onto the $i$-th coordinate of $\mathbb{R}^n$.

Let us see that $\mathrm{span}\{\pi_i\circ G \ : \ i=1,\ldots,n\}$ is contained in $\hat{\mathcal{C}}_0(\mathbb{R}^n)\cup\{0\}$. Given $a_1,\ldots,a_n\in\mathbb{R}$, we have
\[
\left|\sum_{i=1}^{n}a_i (\pi_i\circ G)(x)\right|\leq \left\|\sum_{i=1}^{n}a_i \pi_i\right\|\cdot\|G(x)\|,
\]
which together with $\lim_{\|x\|\rightarrow\infty}\|G(x)\|=0$, gives us that $\sum_{i=1}^{n}a_i (\pi_i\circ G)\in \hat{\mathcal{C}}_0(\mathbb{R}^n)\cup\{0\}$.
\end{proof}

Let us now recall the definition of an alternating function (see \cite{GQ}).

\begin{definition}
Let $A$ be a set with $\mathrm{card}(A)\geq2$. A function $f:A\longrightarrow\mathbb{R}$ is said to be alternating if there exist $x, y\in A$ such that $f(x)<0<f(y)$. A set of functions from $A$ to $\mathbb{R}$ is alternating when each function is alternating. When $V$ is a space of functions $f:A\longrightarrow\mathbb{R}$, we say that $V$ is an alternating space if $V\smallsetminus\{0\}$ is an alternating set.
\end{definition}

The result below will be useful in proving the non-$(n+2)$-lineability of $\hat{\mathcal{C}}_0(\mathbb{R}^n)$. The case $n=1$ is covered in \cite{GQ}.

\begin{proposition}\label{p1}
There does not exist an $(n+1)$-dimensional vector subspace $V$ of $\mathcal{{C}}_0(\mathbb{R}^n)$ such that $V$ is alternating and $V\smallsetminus\{0\}\subset\hat{\mathcal{C}}_0(\mathbb{R}^n)$.
\end{proposition}

\begin{proof}
Suppose there are $f_1,\ldots,f_{n+1}\in\mathcal{{C}}_0(\mathbb{R}^n)$ linearly independent normalized vectors such that $\mathrm{span}\{f_1,\ldots,f_{n+1}\}\smallsetminus\{0\}\subset\hat{\mathcal{C}}_0(\mathbb{R}^n)$ is alternating. Let $Z=S^{n-1}\cap\mathrm{span}\{f_1,\ldots,f_{n+1}\}$ the unitary sphere of $V=\mathrm{span}\{f_1,\ldots,f_{n+1}\}$. Since $V$ has finite dimension, there are $C_1,C_2>0$ such that 
\[
C_1\cdot\left\|\sum_{i=1}^{n+1}a_i f_i\right\|\leq\|(a_1,\ldots,a_{n+1})\|\leq C_2\cdot\left\|\sum_{i=1}^{n+1}a_i f_i\right\| 
\]
for all $(a_1,\ldots,a_{n+1})\in\mathbb{R}^{n+1}$. Then, given $\sum_{i=1}^{n+1}a_i f_i\in Z$, we have $C_1\leq\|(a_1,\ldots,a_{n+1})\|\leq C_2$.

Let us define for each $g=\sum_{i=1}^{n+1}a_i f_i\in Z$, 
\[
m_{(a_1,\ldots,a_{n+1})}=m(g)=\min\left\{\left(\sum_{i=1}^{n+1}a_i f_i\right)(x) \ : \ x\in\mathbb{R}^n\right\}
\]
and
\[
M_{(a_1,\ldots,a_{n+1})}=M(g)=\max\left\{\left(\sum_{i=1}^{n+1}a_i f_i\right)(x) \ : \ x\in\mathbb{R}^n\right\}.
\]
Since $\sum_{i=1}^{n+1}a_i f_i\in Z$ is an alternating function, we have $m_{(a_1,\ldots,a_{n+1})}<0<M_{(a_1,\ldots,a_{n+1})}$.
Let us see that
\[
\sup\left\{m_{(a_1,\ldots,a_{n+1})} \ : \ \sum_{i=1}^{n+1}a_i f_i\in Z\right\}<0.
\]
Indeed, if this were not the case, there would be a sequence $(a^{(k)}_1,\ldots,a^{(k)}_{n+1})_{k\in\mathbb{N}}\subset B_{C_2}\subset\mathbb{R}^{n+1}$ such that
\begin{equation}\label{eq1}
-\frac{1}{k}<m_{(a^{(k)}_1,\ldots,a^{(k)}_{n+1})}<0
\end{equation}
for all $k\in\mathbb{N}$. Since $B_{C_2}\subset\mathbb{R}^{n+1}$ is compact, we can assume without loss of generality that $(a^{(k)}_1,\ldots,a^{(k)}_{n+1})\overset{k\rightarrow\infty}{\longrightarrow}(a_1,\ldots,a_{n+1})\in B_{C_2}$. Thus, we conclude that
\[
\sum_{i=1}^{n+1}a^{(k)}_i f_i\overset{k\rightarrow\infty}{\longrightarrow}\sum_{i=1}^{n+1}a_i f_i\in Z
\]
and
\[
m_{(a^{(k)}_1,\ldots,a^{(k)}_{n+1})}\overset{k\rightarrow\infty}{\longrightarrow}m_{(a_1,\ldots,a_{n+1})}.
\]
Therefore, from \eqref{eq1} we obtain $m_{(a_1,\ldots,a_{n+1})}=0$, however this contradicts the fact that $\sum_{i=1}^{n+1}a_i f_i$ is alternating.

Similarly, we obtain that
\[
\inf\left\{M_{(a_1,\ldots,a_{n+1})} \ : \ \sum_{i=1}^{n+1}a_i f_i\in Z\right\}>0.
\]

Now, let $N>0$ be such that for all $g\in Z$,
\[
m(g)<-N<0<N<M(g).
\]
Such $N$ exists due to the inequalities we have just verified. Since $f_1,\ldots,f_{n+1}\in\hat{\mathcal{C}}_0(\mathbb{R}^n)$ and for $\sum_{i=1}^{n+1}a_i f_i\in Z$, $a_1,\ldots,a_{n+1}\in [-C_2,C_2]$, there exists $A>0$, large enough, so that for any $g=\sum_{i=1}^{n+1}a_i f_i\in Z$, we have $|g(x)|\leq N$ whenever $\|x\|> A$, that is, we have $g(x)<M(g)$ for all $g\in \mathrm{span}\{f_1,\ldots,f_{n+1}\}\smallsetminus\{0\}$ and $x\in\mathbb{R}^n\smallsetminus B_A$. Now considering the subspace $\mathrm{span}\{f_1|_{B_A},\ldots,f_{n+1}|_{B_A}\}$ of $\mathcal{C}(B_A)$, we see that every non-trivial linear combination of $f_1|_{B_A},\ldots,f_{n+1}|_{B_A}$ attains a unique strictly positive maximum point. This tells us that the functions $f_1|_{B_A},\ldots,f_{n+1}|_{B_A}$ are linearly independent and that $\mathrm{span}\{f_1|_{B_A},\ldots,f_{n+1}|_{B_A}\}\smallsetminus\{0\}\subset \hat{\mathcal{C}}(B_A)$. This proves the $(n+1)$-lineability of $\hat{\mathcal{C}}(B_A)$, which is absurd, as it contradicts Theorem \ref{t1}.
\end{proof}

%
%

\begin{proposition} \cite{GQ} \label{p2}
Let $n\geq 2$, and let $A$ be a set. Every $n$-dimensional space of functions $f:A\longrightarrow \mathbb{R}$ contains an alternating space $(n-1)$-dimensional.
\end{proposition}

Let us now prove our main result.

\begin{theorem}
Let $n\geq 2$. The set $\hat{\mathcal{C}}_0(\mathbb{R}^n)$ is $n$-lineable but not $(n+2)$-lineable.
\end{theorem}

\begin{proof}
The $n$-lineability follows immediately from Theorem \ref{c1}. Now, suppose, for the sake of contradiction, that there exists a subspace $W$ of ${\mathcal{C}}_0(\mathbb{R}^n)$, $(n+2)$-dimensional, such that
\begin{equation}\label{eq4}
W\smallsetminus\{0\}\subset \hat{\mathcal{C}}_0(\mathbb{R}^n).
\end{equation}

From Proposition \ref{p2}, $W$ has an alternating subspace $W_0$ of dimension $n+1$. From \eqref{eq4}, it follows that $W_0\smallsetminus\{0\}\subset \hat{\mathcal{C}}_0(\mathbb{R}^n)$, which is a contradiction due to Proposition \ref{p1}.
\end{proof}

A comprehensive examination of the outcomes pertaining to the content presented herein leads us to believe that for $n\geq 2$ the set $\hat{\mathcal{C}}_0(\mathbb{R}^n)$ is not $(n+1)$-lineable.

\begin{conjecture}
Let $n\geq 2$. The set $\hat{\mathcal{C}}_0(\mathbb{R}^n)$ is not $(n+1)$-lineable.
\end{conjecture}

\begin{remark}
The main objective of \cite[Theorem 4]{AB} is to relax the assumptions of \cite[Corollary 2.3]{FPRR}. To achieve this goal, we utilized an improved version, by P. Leonetti et al., of \cite[Theorem 2.5]{BO} (see \cite[Lemma 3.1 and Corollary 3.2]{LRS}). However, in our paper \cite{AB}, we mistakenly stated this improved version, and this mistake carried over to \cite[Theorem 4]{AB}. Below are the correct statements of these results and the correct proof for \cite[Theorem 4]{AB}. We apologize for this, although no other result or application of \cite{AB} was affected by this.

\begin{theorem}\label{aaa} (see \cite[Theorem 2.5]{BO} and \cite[Lemma 3.1 and Corollary 3.2]{LRS})
Let $E$ be a metrizable topological vector space, and let $F$ be a vector subspace of $E$. If $\mathrm{codim}(F)\geq \mathrm{w}(E)$, then $E\smallsetminus F$ is $\mathrm{w}(E)$-dense-lineable\footnote{Let $E$ be a topological vector space and $\alpha$ be a cardinal number. We say that a subset $A$ of $E$ is $\alpha$-dense-lineable whenever $A \cup \{0\}$ contains a dense vector subspace $F$ of $E$ with $\dim(F) = \alpha$.}. Here, $\mathrm{w}(E)$ denotes the weight of $E$, and $\mathrm{codim}(F)$ refers to the algebraic codimension of $F$.
\end{theorem}

\begin{theorem}
Let $E$ be a metrizable topological vector space, and let $F$ be a proper vector subspace of $E$. If $\mathrm{codim}(F)\geq \mathrm{w}(E)$, then $E\smallsetminus F$ is not $(\alpha,\beta)$-spaceable\footnote{Let $E$ be a topological vector space, and let $\alpha$ and $\beta$ be cardinal numbers, with $\alpha<\beta$. We say that $A\subset E$ is $(\alpha,\beta)$-lineable if it is $\alpha$-lineable and for every subspace $F_\alpha\subset E$ with $F_\alpha\subset A\cup\{0\}$ and $\dim(F_\alpha)=\alpha$, there is a closed subspace $F_\beta\subset E$ with $\dim (F_\beta)=\beta$ and $F_\alpha\subset F_\beta \subset A\cup \{0\}$.} for all $\mathrm{w}(E)\leq \alpha \leq \mathrm{codim}(F)$, regardless of the cardinal number $\beta$.
\end{theorem}

\begin{proof}
From Theorem \ref{aaa} we conclude that $E\smallsetminus F$ is $\mathrm{w}(E)$-dense-lineable. From \cite[Theorem 1]{AB} we know that $E\smallsetminus F$ is $(\alpha, \mathrm{codim}(F))$-lineable for all $\alpha < \mathrm{codim}(F)$. In particular, if $\mathrm{w}(E) \leq \alpha < \mathrm{codim}(F)$, the result follows from \cite[Theorem 2]{AB}. If $\mathrm{codim}(F) = \mathrm{w}(E)$, there is no $\beta$-dimensional subspace of $E$ with $\beta > \mathrm{w}(E)$ contained in $(E\smallsetminus F) \cup\{0\}$.
\end{proof}
\end{remark}

\end{document}